\documentclass[reqno, final, a4paper]{amsart}
\usepackage{color}
\usepackage[colorlinks=true,allcolors=blue
]{hyperref}
\usepackage{amsmath, amssymb, amsthm}
\usepackage{mathrsfs}
\usepackage{mathtools}
\usepackage[noabbrev,capitalize,nameinlink]{cleveref}
\crefname{equation}{}{}
\usepackage{fullpage}
\usepackage[noadjust]{cite}
\usepackage{graphics}
\usepackage{pifont}
\usepackage{cancel}
\usepackage{tikz}
\usepackage{bbm}
\usepackage[T1]{fontenc}
\usetikzlibrary{arrows.meta}

\usepackage{environ}
\usepackage{framed}
\usepackage{paralist}
\usepackage{url}
\usepackage[linesnumbered,ruled,vlined]{algorithm2e}
\usepackage[noend]{algpseudocode}
\usepackage[labelfont=bf]{caption}
\usepackage{framed}
\usepackage[framemethod=tikz]{mdframed}
\usepackage{pgfplots}
\usepackage{appendix}
\usepackage{graphicx}
\usepackage[textsize=tiny]{todonotes}
\usepackage{tcolorbox}
\usepackage{xcolor}
\usepackage[shortlabels]{enumitem}
\crefformat{enumi}{#2#1#3}
\crefrangeformat{enumi}{#3#1#4 to~#5#2#6}
\crefmultiformat{enumi}{#2#1#3}
{ and~#2#1#3}{, #2#1#3}{ and~#2#1#3}
\allowdisplaybreaks

\let\originalleft\left
\let\originalright\right
\renewcommand{\left}{\mathopen{}\mathclose\bgroup\originalleft}
\renewcommand{\right}{\aftergroup\egroup\originalright}

\crefname{algocf}{Algorithm}{Algorithms}

\crefname{equation}{}{} 
\AtBeginEnvironment{appendices}{\crefalias{section}{appendix}} 

\usepackage[color]{showkeys} 

\colorlet{refkey}{orange!20}
\colorlet{labelkey}{blue!30}

\crefname{algocf}{Algorithm}{Algorithms}

\numberwithin{equation}{section}

\crefname{claim}{Claim}{Claims}

\newtheorem*{question*}{Question}

\crefname{fact}{Fact}{Facts}

\theoremstyle{definition}
\newtheorem{definition}{Definition}

\newtheorem*{definition*}{Definition}

\theoremstyle{remark}

\newtheorem*{remark*}{Remark}

\newcommand{\eps}{\varepsilon}

\newcommand*{\claimproofname}{Proof of claim}

\usepackage{lineno}
\theoremstyle{plain}
\newtheorem{thm}{Theorem}

\newtheorem{lem}{Lemma}

\newtheorem*{clm*}{Claim}

\theoremstyle{definition}
\newtheorem{rem}{Remark}

\usetikzlibrary{shapes.geometric}
\usepackage{tikz,xcolor,verbatim,hyperref}

\usepackage{euflag}

\title[A note on multicolour Ramsey numbers and random sphere graphs]{A note on multicolour Ramsey numbers \\ and random sphere graphs}
 \author{Yamaan Attwa}
 \author{Albert L\'opez Vidal}
 \author{Patrick Morris}

 \address{YA: Freie Universit\"at Berlin, Germany.  \newline \indent  ALV, PM: Universitat Polit\`ecnica de Catalunya (UPC), Barcelona, Spain.}

\thanks{ YA was supported by the Deutsche Forschungsgemeinschaft (DFG, German Research Foundation) under Germany´s Excellence Strategy – The Berlin
Mathematics Research Center MATH+ (EXC-2046/1, project ID: 390685689).  PM was supported  by  the European Union's Horizon Europe   Marie Sk{\l}odowska-Curie grant RAND-COMB-DESIGN - project number
101106032 {\euflag}. This work was carried out during a visit of the first author to the UPC, funded by the AEI and DFG bilateral research project SRC-ExCo (PCI2024-155080-2). }

\email{y.attwa@fu-berlin.de, lopez.vidal.albert@gmail.com, pmorrismaths@gmail.com}

\date{\today}

\begin{document}

\begin{abstract}
The Ramsey number $r(t;\ell)$ is the smallest  $n$ such that every $\ell$-coloring of the edges of $K_n$ gives a monochromatic $K_{t}$. In recent years, there have been several improvements on asymptotic lower bounds for these numbers when $\ell\geq 3$. This  started with a breakthrough result of Conlon and Ferber, followed by further improvements of Wigderson and then Sawin. Building on the previous approaches, Sawin used blowups of an unbalanced binomial random graph to show that there is some explicit constant $\delta_*\approx 
0.383796$ such that $r(t;\ell)\geq 2^{\delta_*(\ell-2)t+t/2+o(t)}$. In this short note, we show that one can get an exponential improvement in this bound  by replacing the use of a binomial random graph with a random sphere graph, a model which which has recently been applied by Ma, Shen and Xie in a breakthrough on lower bounds for (2-colour) Ramsey numbers in the (slightly) off-diagonal setting. 
\end{abstract}

\maketitle

\vspace{-6mm}
\section{Introduction}

The Ramsey number $r(t;\ell)$ is the smallest positive integer $n$ such that every $\ell$\textit{-coloring} of the edges of the complete graph on $n$ vertices contains a monochromatic copy of $K_{t}$. The most famous of these numbers correspond to the case $\ell=2$ and determining the order of $r(t;2)$ is a major open problem in combinatorics. Despite recent breakthroughs \cite{GUPTA,CAMPOS} exponentially improving upper bounds, there remains a gap between the best known bounds. Indeed our current knowledge places $r(t;2)$ between $2^{t/2(1+o(1))}$ and $(4-\eps)^{t(1+o(1))}$ for $\eps\approx 0.2008$  \cite{GUPTA}. The lower bound has been especially stubborn with the best known approach being to consider an appropriate binomial random graph $G(n,1/2)$ and colour all edges red and all non-edges blue, showing that with positive probability there is no monochromatic clique of size $t$. This  extremely influential idea is due to Erd\H{o}s \cite{erdos1947some} from 1947 (see also \cite{spencer1977asymptotic}).

There has been  more success with lower bounds for  more colours, that is $\ell\geq 3$. By generalising the idea of Erd\H{o}s and considering a random colouring of $K_n$, one gets a lower bound of $r(t;\ell)\geq \ell^{t/2}$. Abbott \cite{ABBOTT} improved this in 1972 for $\ell\geq 4$ to  $r(t;\ell)\geq 2^{\frac{t}{2}\lfloor \frac{\ell}{2} \rfloor}$ using a simple \textit{product construction} idea. 
It was not until Conlon and Ferber \cite{CONLONFERBER} in 2021 that there was an exponential improvement in these lower bounds. Their construction for $\ell=3$ used a (random induced subgraph of) an explicit algebraically defined graph for the first colour, with the remaining edges being coloured uniformly at random. For $\ell=4$, their colouring was similar, with the first two colours being inherited from an algebraic colouring of a graph and the final two being random. 
For larger $\ell$, they used the product construction idea of Abbott. Shortly after, Wigderson \cite{WIGDERSON} showed that one sees a further exponential improvement for $\ell\geq 4$ by avoiding the use of product constructions and having the first $\ell-2$ colours all defined via random homomorphisms to the explicit $K_t$-free graph used by Conlon and Ferber. 
Finally, Sawin \cite{SAWIN} built on Wigderson's approach and showed that one sees yet another exponential improvement (for all $\ell\geq 3$) by replacing the Conlon-Ferber graph by a binomial random graph $G(n,p)$.  He showed that for each $p\in (0,1)$  using $G(n,p)$ leads to a 
lower bound of the form $r(t;\ell)\geq 2^{\delta(p)(\ell-2)t+t/2+o(t)}$ with 
$\delta(p):=-\frac{(4 \log(1-p)-\log(p)) \log(p)}{8 \log(1-p)}$. The expression $\delta(p)$ is 
maximised at  $p_*\approx 0.454997$ with $\delta_*:=\delta(p_*)\approx 0.383796$.

Very recently, there was another remarkable breakthrough in lower bounds for Ramsey numbers. This came from Ma, Shen and Xie \cite{MASHENXIE} (see also \cite{hunter2025gaussian}) who utilised a \textit{random sphere graph} model (Definition \ref{def:random sphere}) to give an exponential improvement on the binomial random graph construction of Erd\H{o}s for a slight variation of the 2-colour Ramsey numbers $r(t;2)$ defined above, where in one of the colours one forbids $K_{Ct}$ for $C\neq 1$, as opposed to $K_t$.  In this note, we show that one can use this random sphere graph in the framework developed in \cite{CONLONFERBER,WIGDERSON,SAWIN} in order to get yet another exponential improvement in $r(t;\ell)$.

\begin{thm}\label{thm: lowerbound}
    There exists an $\varepsilon >0$ such that for every fixed $\ell \geq 3$, 
      $r(t;\ell) \geq 2^{(\delta_*+\varepsilon)(\ell-2)t+ t/2 - o(t)}.$
\end{thm}

\begin{rem} The result in Theorem \ref{thm: lowerbound} was obtained independently and simultaneously\footnote{Their manuscript appeared on arXiv on the $21^{\text{st}}$ January 2026. Whilst ours appeared later, our result also appeared in the Master's thesis \cite{Albertthesis} of the second author, submitted on the $8^{\text{th}}$ January 2026. }  by Campos and Pohoata \cite{campos2026update}. In fact they proved more, showing that for each $p\in (0.42,0.5)$ one can improve the corresponding lower bound of Sawin by replacing $G(n,p)$ with a random sphere graph of the same density. 
    \end{rem}

We close by remarking that the best known upper bound on $r(t;\ell)$ is of the form $ 2^{t\ell \log \ell-\alpha_\ell t }$ obtained recently by Balister, Bollob\'as, Campos, Griffiths, Hurley, Morris, Sahasrabudhe and Tiba \cite{BALISTERETALL}. 

\section{Preliminaries} \label{sec:prelims}

\subsection{Sawin's framework} 

\begin{definition}
    For $t,r\geq 2$, we define $c_{r,t}$ to be the infimum over all $K_t$\textit{-free} graphs of the probability that $r$ independent and uniformly distributed vertices (allowing repetition) form an independent set.
\end{definition}

Sawin \cite{SAWIN} showed that upper bounds on $c_{t,t}$ translate to lower bounds on $r(t;\ell)$.
\begin{lem}[\cite{SAWIN}, Lemma 1 ]\label{lem: sawin}
    For all $\ell, t \geq 3$, we have
    \[r(t;\ell) \geq c_{t,t}^{-\frac{\ell-2}{t}}2^{\frac{t-1}{2}}.\]
\end{lem}

The best known lower bound on $r(t;\ell)$ is then obtained from the  following upper bound on $c_{t,t}$.
\begin{lem}[\cite{SAWIN}, Lemma 2] \label{lem: ctt} For any $p \in (0,1)$ and $t\in \mathbb{N}$ we have \begin{equation}\label{eq: ctt}
    c_{t,t} \leq 2^{\frac{t^2(4 \log(1-p) - \log (p) )\log (p)}{8 \log (1-p)}+t \log t + \mathcal{O}(t)}.
\end{equation} 
\end{lem}
Combining Lemmas \ref{lem: sawin} and \ref{lem: ctt} thus gives that for any $p\in (0,1)$, $r(t;\ell)\geq 2^{\delta(p)(\ell-2)t+t/2+o(t)}$ with
$\delta(p):=-\frac{(4 \log(1-p)-\log(p)) \log(p)}{8 \log(1-p)}$. For the remainder of the paper, we fix   $p_*\approx 0.454997$ to be the value of $p$ which maximises
$\delta(p)$ and fix $\delta_*:=\delta(p_*)\approx 0.383796$.
 
   \subsection{Random sphere graphs}
For a natural number $k$, let $S^k$ denote the $k$\textit{-dimensional} unit sphere embedded in $\mathbb{R}^{k+1}.$ Let $p \in (0,1/2)$; if $e \in S^k$ is any fixed arbitrary unit vector and $x$ is a unit vector chosen uniformly at random from $S^k$, we denote by $c^{k,p}>0$ the unique solution to the probabilistic equation \begin{equation}
    \mathbb{P}\Bigl( \langle x,e \rangle \leq \frac{-c^{k,p}}{\sqrt{k}}\Bigr) = p.
\end{equation} 
   Furthermore, we denote by $c_k:= c^{k,p_*} $. 
    \begin{lem}[\cite{MASHENXIE}, Lemma 4.1]\label{lem: c_kConverge} For any fixed $p \in (0,1/2)$, there exists some $c(p)>0$ such that: 
   \[c^{k,p}= c(p) + \mathcal{O}\Bigl(\frac{1}{k}\Bigl).\]
       In particular, there exists some $c_*>0$ such that $\lim_{k \rightarrow \infty}c_k:= \lim _{k \rightarrow\infty}c^{k,p_*}=c_*$.
   \end{lem}

   \begin{definition} \label{def:random sphere}
       Given $n,k \geq 1$ and $p\in (0,1/2)$, we define the random graph model $G_{k,p}(n)$ as the graph $G$ on $n$ vertices sampled independently and uniformly at random from $S^k$ where any pair of vertices $u,v$ forms an edge iff $\langle u,v\rangle\leq \frac{-c^{k,p}}{\sqrt{k}}$.
   \end{definition}
   To ease the notation we define $G_k(n):=G_{k,p_*}(n)$ and we sometimes consider $V(G_k(n))=[n]$.

   \begin{lem}[\cite{MASHENXIE}, Theorem 7.1 and proof of Lemma 9.2]\label{lem: probabilities} There exists some $D_0,C >0$ \footnote{As clarified at the start of Section 6 in \cite{MASHENXIE}, the constant $C>0$ we use here (that is hidden in an $\mathcal{O}(1/D^2)$ term in \cite{MASHENXIE}) is independent from $D$ but can depend on $c_k$ (which in turn depends on $k$ and therefore $t$). However, as we take $t$ (and hence $k$) to be large, by Lemma \ref{lem: c_kConverge} we can assume that $c_k$ is close to $c_*$ and thus the constant $C>0$ used here can be taken to be the maximum of all of the implied constants in \cite{MASHENXIE} over choices of $c_k$ in a small interval around $c_*$. }  such that the following holds: If $n,t \in \mathbb{N}$ are sufficiently large, $D\geq D_0$, $k=D^2t^2$, $G \sim G_{k}(n)$, and $2\leq r \leq t$, then 
   \begin{align*}
       \mathbb{P}([r] \text{ is a clique in $G$}) &\leq (1+2^{-r}) \Biggl( p_* - \Bigl( \frac{e^{-c^2_{k}}}{2\pi}\Bigr)^{3/2}\frac{r-2}{3p_*^2\sqrt{k}}+ \frac{C}{D^2}\Biggr)^{\binom{r}{2}}, \text{ and } \\
       \mathbb{P}([r] \text{ is ind. in $G$}) &\leq   (1+2^{-r}) \Biggl(1- p_* + \Bigl( \frac{e^{-c^2_{k}}}{2\pi}\Bigr)^{3/2}\frac{r-2}{3(1-p_*)^2\sqrt{k}}+ \frac{C}{D^2}\Biggr)^{\binom{r}{2}}.
   \end{align*}
   \end{lem}

   \section{Proof of Theorem \ref{thm: lowerbound}}
   We will prove the following improvement to Lemma \ref{lem: ctt} when $p=p_*$. 

   \begin{thm} \label{thm: main}
       There exists some $\varepsilon>0$ such that \begin{equation}\label{eq: main} 
       c_{t,t} \leq 2^{-(\delta_*+\varepsilon)t^2 +o(t^2)}.
       \end{equation}
   \end{thm}
   Note that Theorem \ref{thm: lowerbound} follows immediately from Theorem \ref{thm: main} combined with  Lemma \ref{lem: sawin}. 
   \begin{proof} Let $C,D_0>0$ be the constants from Lemma \ref{lem: probabilities} and for $D>D_0$ (to be determined later) take $k=D^2t^2$ and $G \sim G_k(M)$. Define $a_k:= \frac{1}{3} \Bigl(\frac{e^{-c_k^2}}{2\pi}\Bigr)^{3/2}$ and 
   \[M= \Bigl(p_*-  \frac{a_k}{Dp_*^2}+ \frac{C}{D^2}\Bigr)^{-t/2} .\] 
    Let $v_1, . . . , v_t$ be uniform random variables over $[M]$ independent from each other and from $G$, then \begin{equation}\label{eq: cttfraction}
       c_{t,t} \leq \frac{\mathbb{P}\Bigl( \{v_1,\dots,v_t\} \text{ is an ind. set in $G$}\Bigr)}{\mathbb{P}\Bigl(G \text{ contains no clique of size $t$}  \Bigr) }.
   \end{equation}
   Indeed, the right-hand side is an upper bound for the expectation over $G$ of the probability that $\{v_1, . . . , v_t\}$ is independent conditional on $G$ containing no $K_t$. Since every random variable has an instance at which its value is at most its expectation, (\ref{eq: cttfraction}) follows. If $F$ is the denominator, using Lemma \ref{lem: probabilities}, the fact that $t/\sqrt{k}= 1/D$ and a union bound we get that that the probability $t$ random vertices in $G$ form a clique is
   \begin{align}
      1-F \leq &\binom{M}{t} (1+2^{-t}) \Biggl( p_* - \frac{a_k}{Dp_*^2}\cdot\frac{t-2}{t}+ \frac{C}{D^2}\Biggr)^{\binom{t}{2}}\\
       &\leq \frac{(1+2^{-t})}{t!}M^t  \Biggl( p_* - \frac{a_k}{Dp_*^2}+ \frac{C}{D^2}+\frac{2a_k}{ D p_*^2 t}\Biggr)^{\binom{t}{2}}.
   \end{align}
   Note that $\frac{2a_k}{Dp_*^2}$ and $p_*-\frac{a_k}{Dp_*^2}+\frac{C}{D^2}$ are bounded by constants independent of $t$ in light of Lemma \ref{lem: c_kConverge}. Therefore, there exists a constant $C_1$ such that
   \begin{equation}
       1-F \leq \frac{1+2^{-t}}{t!}e^{\frac{C_1t}{2}} = o(1),
   \end{equation}
   where we have used our choice of $M$, the identity $e^x\geq 1+x$, and that $t!\geq 2^{t\log(t)-O(t)}$. This implies that the denominator in (\ref{eq: cttfraction}) is $1-o(1)$ and we can divert our attention to the numerator. Let $R:=\{v_1, \dots, v_t\}$ and note that $|R|$ may well be below $t$ as repetition is allowed in the choice of the $v_i$, $i\in[t]$. By the law of total probability 
   \begin{align}
   &\mathbb{P}\bigl(R\text{ is ind. in $G$}\bigr)= \sum_{S \in 2^{[M]}} \mathbb{P}\Bigl( R = S\Bigr) \mathbb{P}\Bigl( \text{ $S$ is ind. in $G$ }\:\:\Big|\:\: R=S\Bigr)  \\
   &\leq \sum_{r=1}^t \binom{M}{r} \frac{{t\brace r}r!}{M^t}(1+2^{-t}) \left(1- p_* +  a_k\frac{r-2}{(1-p_*)^2\sqrt{k}}+ \frac{C}{D^2}\right)^{\binom{r}{2}},  
 \end{align}
 where ${t\brace r}$ is the Stirling number of the second kind. Indeed; if $|R|=r$, there are $\binom{M}{r}$ ways to choose the set $S$ where the elements of $\{v_1,\dots,v_t\}$ get mapped, ${t\brace r}$ ways to distribute the labelled random variables $v_1,\dots,v_t$ into $r$ indistinguishable fibres, and $r!$ ways to biject these fibres into the elements of $S$. Simplifying the terms further we get:
 \begin{equation}
     \label{eq: optexponent}
   \mathbb{P}\bigl( R \text{ is ind. in $G$}\bigr)\leq  2t! \underset{1 \leq r \leq t}{\max}2^{\frac{r(r-1)}{2} \log \bigl(1-p_*+ \frac{a_k}{D(1-p_*)^2}+\frac{C}{D^2}\bigr) + \frac{(t-r)t}{2} \log \bigl(p-\frac{a_k}{Dp_*^2}+\frac{C}{D^2}\bigr)},
 \end{equation}
 where we have used that $\sum_{r=0}^{r=t}{t \brace r} \leq t!$ and that $t/\sqrt{k}= 1/D$.
 Note that $t! \leq 2^{t \log t}$ can absorbed in the $2^{o(t^2)}$ error term. The exponent in (\ref{eq: optexponent}) is quadratic in $r$ and is maximized at \begin{equation}
     r=\frac{1}{2} + \frac{t \log \Bigl(p_*-\frac{a_k}{p_*^2D}+\frac{C}{D^2}\Bigr)} {2 \log(1-p_*+\frac{a_k}{D(1-p_*)^2}+\frac{C}{D^2})}.
     \end{equation}
     Plugging in this value for $r$ we get that the exponent in (\ref{eq: optexponent}) is at most 
      \begin{equation} \label{eq: expression}\frac{t^2\log\bigl(p_*-\frac{a_k}{Dp_*^2}+\frac{C}{D^2}\bigr) \Biggl( 4 \log \biggl(1-p_*+\frac{a_k}{D(1-p_*)^2}+\frac{C}{D^2}\biggr)-\log\biggl(p_*-\frac{a_k}{Dp_*^2}+\frac{C}{D^2}\biggr)\Biggr)}{8 \log \bigl(1-p_*+\frac{a_k}{D(1-p_*)^2}+\frac{C}{D^2}\bigr)}  + \mathcal{O}(t)\end{equation}

In light of Lemma \ref{lem: c_kConverge}, there exists some absolute constant $a>0$ such that for every $\eta >0$, if $t$ (and hence $k$) is sufficiently large then $a^-:=(1-\eta)a\leq a_k \leq (1+\eta)a=:a^+$. So for $k$ large we have that the expression in (\ref{eq: expression}) is at most
\begin{equation} \label{eq: expression}\frac{t^2\log\bigl(p_*-\frac{a^-}{Dp_*^2}+\frac{C}{D^2}\bigr) \Biggl( 4 \log \biggl(1-p_*+\frac{a^+}{D(1-p_*)^2}+\frac{C}{D^2}\biggr)-\log\biggl(p_*-\frac{a^+}{Dp_*^2}+\frac{C}{D^2}\biggr)\Biggr)}{8 \log \bigl(1-p_*+\frac{a^-}{D(1-p_*)^2}+\frac{C}{D^2}\bigr)}  + \mathcal{O}(t)\end{equation}
We introduce a formal variable $x$ in place of $1/D$ and define the function:

\begin{equation}
f(x) = -\frac{\log\bigl(p_*-\frac{a^-}{p_*^2}x+Cx^2\bigr) \Biggl( 4 \log \biggl(1-p_*+\frac{a^+}{(1-p_*)^2}x+Cx^2\biggr)-\log\biggl(p_*-\frac{a^+}{p_*^2}x+Cx^2\biggr)\Biggr)}{ 8\log \bigl(1-p_*+\frac{a^-}{(1-p_*)^2}x+Cx^2\bigr)}.
\end{equation}
Note that $\delta_*= f(0)$. A direct calculation shows that: \begin{equation}
    f'(0)=  \frac{a}{2p_*^3}-\frac{a\log(p_*)}{4p_*^3 \log(1-p_*)}-\frac{ a\log^2(p_*)}{8(1-p_*)^3 \log^2(1-p_*)} -\eta \cdot h(p_*,a),
\end{equation} 
for $h$ some function of $p_*$ and $a$.  
If $g$ is the limit of $f'(0)$ as $k$ approaches $\infty$, we have $g \approx 0.565a$. This means that for $t$ (and hence $k$) sufficiently large, $f'(0) \geq 0.56a>0.$ Hence, by appealing to the definition of a derivative, there exists some $\gamma >0$ small enough such that for all $0<x \leq \gamma$, we have $f(x) \geq\delta_*+0.5ax$. If $D_1 = 1/\gamma$, then for all $D\geq D_0,D_1$, we have $f(1/D)\geq \delta_* +0.5a/D > \delta_*$. Setting $D = \max\{D_0,D_1\}$ and $\varepsilon = 0.5a/D$, we conclude that 
\[c_{t,t} \leq 2^{-(\delta_*+\varepsilon)t^2+o(t^2)}.\]
\end{proof}

 \bibliographystyle{abbrv}
 \bibliography{mybib.bib}
\end{document}